\documentclass[11pt]{amsart}
\usepackage{amssymb, amsmath, latexsym,   bm, mathrsfs}
\usepackage{multirow, xcolor, float}
\usepackage{fullpage}
\renewcommand{\baselinestretch}{\baselinestretch}
\renewcommand{\baselinestretch}{1.1}
\numberwithin{equation}{section}

\newtheorem{thm}{Theorem}[section]
\newtheorem{lem}[thm]{Lemma}

\newtheorem{prop}[thm]{Proposition}

\theoremstyle{remark}
\newtheorem{rmk}[thm]{Remark}

\newcommand{\ra}{{\,\rightarrow\,}}
\newcommand{\nra}{{\,\nrightarrow\,}}
\newcommand{\gen}{\text{gen}}

\newcommand{\z}{{\mathbb Z}}
\newcommand{\q}{{\mathbb Q}}

\newcommand{\bx}{\bm x}

\newcommand{\be}{\bm e}

\newcommand{\0}{\bm 0}

\newcommand{\well}{\widetilde{\ell}}

\newcommand{\Mod}[1]{\ (\mathrm{mod}\ #1)}

\begin{document}

\title{Isolations of the sum of two squares from its proper subforms}

\author[Jangwon Ju et al.] {Jangwon Ju, Daejun Kim, Kyoungmin Kim, Mingyu Kim, and Byeong-Kweon Oh}

\address{Department of Mathematics Education, Korea National University of Education, Cheongju 28173, Korea}
\email{jangwonju@knue.ac.kr}
\thanks{This work of the first author was supported by the National Research Foundation of Korea(NRF) grant funded by the Korea government(MSIT) (NRF-2022R1A2C1092314).}

\address{School of Mathematics, Korea Institute for Advanced Study, Seoul 02455, Korea}
\email{dkim01@kias.re.kr}
\thanks{This work of the second author was supported by a KIAS Indivisual Grant (MG085501) at Korea Institute for Advanced Study.}

\address{Department of Mathematics, Hannam University, Daejeon 34430, Korea}
\email{kiny30@hnu.kr}
\thanks{This work of the third author was supported by the National Research Foundation of Korea(NRF) grant funded by the Korea government(MSIT) (NRF-2022R1F1A1063399)	}

\address{Department of Mathematics, Sungkyunkwan University, Suwon 16419, Korea}
\email{kmg2562@skku.edu}
\thanks{This work of the fourth author was supported by the National Research Foundation of Korea(NRF) grant funded by the Korea government(MSIT) (NRF-2021R1C1C2010133).}

\address{Department of Mathematical Sciences and Research Institute of Mathematics, Seoul National University, Seoul 08826, Korea}
\email{bkoh@snu.ac.kr}
\thanks{This work of the fifth author was supported by the National Research Foundation of Korea(NRF) grant funded by the Korea government(MSIT) (NRF-2019R1A2C1086347 and NRF-2020R1A5A1016126).}



\begin{abstract} For a (positive definite and integral) quadratic form $f$, a quadratic form is said to be {\it an isolation of $f$ from its proper subforms} if it represents all proper subforms of $f$, but not $f$ itself. 
It was proved that the minimal rank of isolations of the square quadratic form $x^2$ is three, and there are exactly $15$ ternary diagonal isolations of $x^2$. Recently, it was proved that any quaternary quadratic form cannot be an isolation of the sum of two squares $I_2=x^2+y^2$, and there are quinary isolations of $I_2$.
In this article, we prove that there are at most $231$ quinary isolations of $I_2$, which are listed in Table $1$. Moreover, we prove that $14$ quinary quadratic forms with dagger mark in Table $1$ are isolations of $I_2$. 
 \end{abstract}

\maketitle

\section{Introduction}

For a positive integer $n$, an integral quadratic form $f$ with $n$ variables is a homogeneous quadratic polynomial 
$$
f(x_1,x_2,\dots,x_n)=\sum_{i,j=1}^n f_{ij}x_ix_j \quad (f_{ij}=f_{ji} \in \z).
$$
The symmetric matrix $M_f=(f_{ij})$ is called the Gram matrix corresponding to the quadratic form $f$. Throughout this article, we always assume that any quadratic form $f$ is {\it integral and positive definite}, that is, the corresponding Gram matrix is integral and positive definite. The number $n$  of variables of $f$ is called the rank of $f$.
We call $f$ diagonal if its Gram matrix is a diagonal matrix. For simplicity, the diagonal matrix with entries $a_1, \ldots, a_n$ on its main diagonal will be denoted by $\langle a_1, \ldots, a_n\rangle$.  
Let $g(y_1,y_2,\dots,y_m)=\sum_{i,j=1}^m g_{ij}y_iy_j$ be another quadratic form of rank $m$, and let $M_g=(g_{ij})$ be the Gram matrix corresponding to $g$. We say $g$ is represented by $f$ if there is an integral matrix  $T=(t_{ij}) \in M_{n,m}(\z)$ such that 
$$
T^tM_fT=M_g.
$$ 
Hence the existence of a representation between two quadratic forms is equivalent to the existence of an integral solution of the system of diophantine equations given by those quadratic forms. Any quadratic form that is represented by the quadratic form $f$ is called a subform of $f$. 

Let $g$ and $f$ be quadratic forms. If $g$ is represented by $f$, then clearly, every proper subform of $g$ is also represented by $f$. 
However, the converse is not true in general. For example, it was proved in \cite{jko} that the diagonal ternary quadratic form $f=2x^2+2y^2+5z^2$ represents all squares of integers except for $1$, that is, $f$ represents all proper subforms of the unary quadratic form $x^2$, but not $x^2$ itself. On the contrary,  Elkies, Kane, and Kominers proved in \cite{ekk} that any quadratic form which represents all proper subforms of $x^2+y^2+2z^2$ represents $x^2+y^2+2z^2$ itself.  

A (positive definite and integral) quadratic form is called  {\it an isolation} of a quadratic form $f$ if it represents all proper subforms  of $f$, but not $f$ itself.  If there is an isolation of $f$,  then the minimum rank of isolations of the quadratic form $f$ is denoted by $\text{Iso}(f)$.  As stated above, the diagonal ternary quadratic form $2x^2+2y^2+5z^2$ is an isolation of $x^2$, and one may easily check that $\text{Iso}(x^2)=3$. For more information on quadratic forms having an isolation, one may see \cite{co} and \cite{klo}.
Recently, the fifth author proved in \cite{o} that $\text{Iso}(x^2+y^2)=5$, and in fact, two quinary quadratic forms having Gram matrices both
$$
\begin{pmatrix} 1&0&0&0&0\\0&2&0&1&0\\0&0&2&1&0\\0&1&1&2&0\\0&0&0&0&3\end{pmatrix} \quad \text{and} \quad  \begin{pmatrix} 1&0&0&0&0\\0&2&0&0&0\\0&0&2&1&0\\0&0&1&2&1\\0&0&0&1&3\end{pmatrix} 
$$
are isolations of $x^2+y^2$.  Furthermore, he proved that $\text{Iso}(x^2+y^2+z^2)=6$ and the senary quadratic form $x^2+y^2+2z^2+2u^2+2v^2+2vu+2vz+2zu+3w^2$ is an isolation of $x^2+y^2+z^2$. 

In this article, we prove that there are at most $231$ quinary isolations of $x^2+y^2$. Among them, we prove that $14$ particular quinary quadratic forms with dagger mark in Table \ref{table1} are, in fact, isolations of $x^2+y^2$.

The subsequent discussion will be conducted in the language of quadratic spaces and lattices. For simplicity, the quadratic map and its associated bilinear form on any quadratic space will be denoted by $Q$ and $B$, respectively. The term {\em lattice} always means a finitely generated $\z$-module on a finite dimensional positive definite quadratic space over $\q$.  

Let $L=\z\bx_1+\z\bx_2+\dots+\z\bx_n$ be a $\z$-lattice of rank $n$. The corresponding symmetric matrix $M_L=(B(\bx_i,\bx_j))$ is called the Gram matrix of $L$. We will often address a positive definite symmetric matrix as a lattice, and write $L=M_L$.
For a prime $p$, let $\z_p$ be the $p$-adic integer ring. We define $L_p=L\otimes \z_p$, which is considered as  a $\z_p$-lattice.  A $\z$-lattice $M$ is said to be represented by  $L$ if there is a linear map $\sigma: M \rightarrow L$ such that $Q(\sigma(\bx)) = Q(\bx)$ for any $\bx \in M$.  Such a map is called a representation from $M$ into $L$, which is necessarily injective because the symmetric bilinear map defined on $M$ is assumed to be nondegenerate.   If there is a linear map $\sigma_p : M_p \ra L_p$ satisfying the above property for some prime $p$, then we say $M$ is represented by $L$ over $\z_p$. Furthermore, we say $M$ is locally represented by $L$ if $M$ is represented by $L$ over $\z_p$ for any prime $p$.  If $M$ is represented by $L$, then we simply write $M \ra L$. In particular, if $M=\langle m\rangle$ is a unary $\z$-lattice, then we write $m \ra L$ as well as $\langle m\rangle \ra L$.  

Two $\z$-lattices $L$ and $M$ are isometric if there exists a representation sending $L$ onto $M$.  In this case we will write $L \simeq M$. If $M$ is isometric to $L$ over $\z_p$ for any prime $p$, then we say $M$ is contained in the genus of $L$, and we write $M \in \gen(L)$.  The number of isometry classes in the  genus of $L$ is called the class number of $L$, and is denoted by $h(L)$. It is well known that the class number of any $\z$-lattice is always finite. It is also well known that a $\z$-lattice $K$ is locally represented by $L$ if and only if there is a $\z$-lattice $M \in \gen(L)$ such that $K \ra M$. In this case, we write $K\ra \gen(L)$. Meanwhile, as far as the authors know, there is no known effective way how to find such a $\z$-lattice $M \in \gen(L)$ that represents $K$.  

If $L$ and $M$ are $\z$-lattices, their orthogonal sum is denoted by $L \perp M$. For a positive integer $n$, the $\z$-lattice $I_n=\z \be_1+\dots+\z \be_n$ is called the cubic lattice of rank $n$ if its Gram matrix $(B(\be_i,\be_j))$ is the identity matrix of rank $n$. A binary $\z$-lattice $\ell=\begin{pmatrix} a&b\\b&c\end{pmatrix}$ is called Minkowski reduced if $0\le 2b\le a\le c$. 
Let $\mathbb{H}$  and $\mathbb{A}$ be  binary $\z$- or $\z_p$-lattices whose Gram matrices are 
$$
\mathbb{H}=\begin{pmatrix}0&1\\1&0\end{pmatrix} \quad \text{and} \quad  \mathbb{A}=\begin{pmatrix}2&1\\1&2\end{pmatrix}.
$$

 The readers are referred to \cite{ki} and \cite{om2} for any unexplained notations and terminologies.


\section{Candidates of quinary isolations of $I_2$}\label{section2}
In this section, we find all candidates of quinary isolations of $I_2$. For a $\z$-lattice $L$ of rank $k$ and an integer $i$ with $1\le i \le k$, let $\mu_i(L)$ be the $i$-th successive minimum of $L$. The following lemma enables us to bound the successive minima of quinary isolations of $I_2$.

\begin{lem} \label{lemmu}
	Let $L=\z \bx_1+\z \bx_2+\cdots+\z \bx_5$ be a quinary $\z$-lattice, where $\{ \bx_i\}_{i=1}^{5}$ is a Minkowski reduced basis such that $Q(\bx_1)\le \cdots \le Q(\bx_5)$. If a binary $\z$-lattice $\ell$ is represented by $L$ but is not represented by the $k\times k$ section $\z \bx_1+\z \bx_2+\cdots+\z \bx_k$ of $L$, then
	$$
	Q(\bx_{k+1})\le \begin{cases}\mu_2(\ell)&\text{if}\ 1\le k\le 3,\\
		\dfrac{5}{4}\mu_2(\ell)&\text{if}\ k=4.\end{cases}
	$$	
\end{lem}

\begin{proof}
	The lemma follows immediately from Lemma 2.1 of \cite{loy} and Theorem 3.1 of Chapter $12$ of \cite{c}.
\end{proof}


\begin{scriptsize}
	\begin{table}[h]
		\caption{Candidates for quinary isolations of $I_2$}
		\label{table1}
		\begin{tabular}{c|l|c}
			\hline
			$(a_1,a_2,a_3)$ & $(b_1,b_2,b_3,b_4)$ & \# \\
			\hline
			$(1,1,2)$ & $(1,0,0,3)^{\dag}$, $(0,0,0,3)^{\ast}$, $(1,-1,0,5)$ & 3\\
			\hline
			\multirow{3}{*}{$(0,1,2)$}
			& $(1,1,0,3)^{\dag}$, $(0,0,1,3)^{\ast}$, $(1,0,0,3)$, $(0,0,0,3)^{\dag}$, $(1,1,0,5)$, & \multirow{3}{*}{15}\\
			& $(0,1,0,5)$, $(1,0,0,5)$, $(0,0,0,5)^{\dag}$, $(1,1,0,7)$, $(0,1,0,7)$,\\
			& $(1,0,0,7)$, $(0,0,0,7)$, $(1,1,0,9)$, $(0,1,0,9)$, $(0,0,0,9)^{\dag}$\\
			\hline
			$(0,0,2)$ & $(0,1,1,3)^{\dag}$, $(0,0,1,3)^{\dag}$, $(0,1,1,5)$, $(1,0,0,5)$ & 4\\
			\hline
			$(1,1,3)$ & $(1,-1,0,4)$, $(1,0,0,4)$, $(1,-1,1,5)$, $(1,-1,0,5)$ & 4\\
			\hline
			\multirow{12}{*}{$(0,1,3)$}
			& $(1,1,0,3)$, $(1,0,1,3)$, $(0,1,0,3)^{\dag}$, $(1,0,0,3)^{\dag}$, $(0,0,1,3)$, & \multirow{12}{*}{54}\\
			& $(1,1,0,4)$, $(0,1,0,4)$, $(1,0,0,4)^{\dag}$, $(0,0,1,4)$, $(1,1,0,5)$,\\
			& $(0,0,0,4)$, $(1,0,1,5)$, $(0,1,0,5)$, $(0,0,1,5)$, $(1,1,0,6)$,\\
			& $(1,0,1,6)$, $(1,0,0,6)$, $(0,0,1,6)$, $(1,1,0,7)$, $(0,0,0,6)$,\\
			& $(1,0,1,7)$, $(0,1,0,7)$, $(1,0,0,7)$, $(0,0,1,7)$, $(1,1,0,8)$,\\
			& $(0,0,0,7)$, $(1,0,1,8)$, $(0,0,1,8)$, $(0,0,0,8)$, $(1,0,1,9)$,\\
			& $(0,1,0,9)$, $(1,0,0,9)$, $(0,0,1,9)$, $(1,1,0,10)$, $(0,0,0,9)^{\dag}$,\\
			& $(1,0,1,10)$, $(0,1,0,10)$, $(1,0,0,10)$, $(0,0,1,10)$, $(1,1,0,11)$,\\
			& $(0,0,0,10)$, $(1,0,1,11)$, $(0,1,0,11)$, $(1,0,0,11)$, $(0,0,1,11)$,\\
			& $(1,1,0,12)$, $(1,0,1,12)$, $(0,1,0,12)$, $(1,0,0,12)$, $(0,0,1,12)$,\\
			& $(1,1,0,13)$, $(1,0,1,13)$, $(0,1,0,13)$, $(1,0,0,13)$\\
			\hline
			$(0,0,3)$ & $(1,0,0,4)^{\dag}$, $(1,1,0,5)$, $(1,0,0,5)^{\dag}$ & 3\\
			\hline
			$(1,1,4)$ & $(1,-1,2,5)$, $(1,-1,0,5)$, $(1,0,0,5)$, $(1,-1,2,7)$, $(0,0,0,5)^{\dag}$ & 5\\
			\hline
			\multirow{5}{*}{$(0,1,4)$}
			& $(0,-1,2,5)$, $(1,1,2,5)$, $(0,1,2,5)$, $(1,1,0,5)$, $(1,0,1,5)$, & \multirow{5}{*}{23}\\
			& $(0,1,0,5)$, $(1,0,0,5)$, $(0,0,1,5)$, $(0,0,0,5)$, $(1,1,2,7)$,\\
			& $(0,1,2,7)$, $(1,1,0,7)$, $(1,0,1,7)$, $(0,1,0,7)$, $(1,0,0,7)$,\\
			& $(0,0,1,7)$, $(1,1,2,9)$, $(0,1,2,9)$, $(1,1,0,9)$, $(1,0,1,9)$,\\
			& $(0,1,0,9)$, $(1,0,0,9)$, $(0,0,0,9)$ \\
			\hline
			$(0,0,4)$ & $(1,1,0,5)$, $(1,0,0,5)$ & 2\\
			\hline
			\multirow{23}{*}{$(1,1,5)$}
			& $(-1,-1,2,5)$, $(-1,-1,2,6)$, $(1,-1,1,5)$, $(-1,-1,2,7)$, & \multirow{23}{*}{111}\\
			& $(1,-1,0,5)$, $(1,-1,2,6)$, $(1,-1,1,6)$, $(0,0,1,5)$, $(1,0,2,6)$,\\
			& $(1,-1,0,6)$, $(1,-1,2,7)$, $(0,0,0,5)$, $(1,0,0,6)$, $(1,-1,1,7)$,\\
			& $(0,0,1,6)$, $(1,0,2,7)$, $(1,-1,2,8)$, $(1,0,0,7)$, $(1,-1,1,8)$,\\
			& $(0,0,1,7)$, $(1,-1,2,9)$, $(0,0,0,7)$, $(1,0,0,8)$, $(1,-1,1,9)$,\\
			& $(1,0,2,9)$, $(1,-1,0,9)$, $(1,-1,2,10)$, $(0,0,0,8)$, $(1,0,0,9)$,\\
			& $(1,-1,1,10)$, $(0,0,1,9)$, $(1,0,2,10)$, $(0,0,0,9)$, $(1,0,0,10)$,\\
			& $(0,0,1,10)$, $(1,0,2,11)$, $(1,0,0,11)$, $(1,-1,1,12)$, $(0,0,1,11)$,\\
			& $(1,0,2,12)$, $(1,-1,0,12)$, $(1,-1,2,13)$, $(1,0,0,12)$, $(1,-1,1,13)$,\\
			& $(1,0,2,13)$, $(1,-1,0,13)$, $(1,-1,2,14)$, $(0,0,0,12)$, $(1,0,0,13)$,\\
			& $(1,-1,1,14)$, $(0,0,1,13)$, $(1,0,2,14)$, $(1,-1,0,14)$, $(1,0,0,14)$,\\
			& $(0,0,1,14)$, $(1,0,2,15)$, $(1,-1,2,16)$, $(1,-1,1,16)$, $(0,0,1,15)$,\\
			& $(1,0,2,16)$, $(1,-1,2,17)$, $(0,0,0,15)$, $(1,0,0,16)$, $(1,-1,1,17)$,\\
			& $(1,0,2,17)$, $(1,-1,0,17)$, $(1,-1,2,18)$, $(0,0,0,16)$, $(1,0,0,17)$,\\
			& $(1,-1,1,18)$, $(0,0,1,17)$, $(1,0,2,18)$, $(1,-1,0,18)$, $(0,0,0,17)$,\\
			& $(1,0,0,18)$, $(0,0,1,18)$, $(1,0,2,19)$, $(1,0,0,19)$, $(1,-1,1,20)$,\\
			& $(1,0,2,20)$, $(1,-1,0,20)$, $(1,-1,2,21)$, $(1,0,0,20)$, $(1,-1,1,21)$,\\
			& $(1,0,2,21)$, $(1,-1,0,21)$, $(1,-1,2,22)$, $(0,0,0,20)$, $(1,0,0,21)$,\\
			& $(1,-1,1,22)$, $(0,0,1,21)$, $(1,0,2,22)$, $(1,-1,0,22)$, $(0,0,0,21)$,\\
			& $(1,0,0,22)$, $(0,0,1,22)$, $(1,0,2,23)$, $(1,-1,2,24)$, $(1,0,0,23)$,\\
			& $(1,-1,1,24)$, $(0,0,1,23)$, $(1,0,2,24)$, $(1,-1,2,25)$, $(0,0,0,23)$,\\
			& $(1,-1,1,25)$, $(1,0,2,25)$, $(1,-1,0,25)$, $(0,0,2,25)$, $(1,0,0,25)$,\\
			& $(0,0,1,25)$, $(0,0,0,25)$ \\
			\hline
			\multirow{2}{*}{$(0,1,5)$}
			& $(1,1,2,5)$, $(1,1,0,5)$, $(0,0,2,5)$, $(0,1,0,5)$, $(1,0,0,5)$, & \multirow{2}{*}{7}\\
			& $(0,0,1,5)$, $(0,0,0,5)$ \\
			\hline
		\end{tabular}
	\end{table}
\end{scriptsize}


Let $L$ be a quinary isolation of $I_2$ with a Minkowski reduced basis $\{ \bx_i\}_{i=1}^{5}$.
Since $L$ represents the sublattice $\langle 1,4\rangle$ of $I_2$, we have $Q(\bx_1)=\mu_1(L)=1$.
Similarly, as $L$ represents $\langle 2,2\rangle$, it follows from Lemma \ref{lemmu} that $Q(\bx_2)=\mu_2(L)\le 2$.
Hence we have $\mu_2(L)=2$.
Noting that $\langle 2,2\rangle \nra \langle 1,2\rangle$, we have $\mu_3(L)=2$ by Lemma \ref{lemmu}.
Thus 
$$
L(3)=\z \bx_1+\z \bx_2+\z \bx_3\simeq \langle 1,2,2\rangle \ \ \text{or}\ \ \langle 1\rangle \perp \mathbb{A}.
$$
Noting further that
$$
\begin{pmatrix}2&1\\1&5\end{pmatrix} \nra \langle 1,2,2\rangle\ \ \text{and}\ \ \langle 1,4\rangle \nra \langle 1\rangle \perp \mathbb{A},
$$
Lemma \ref{lemmu} again implies that
$$
\mu_4(L)\le \begin{cases}5&\text{if}\ L(3)\simeq \langle 1,2,2\rangle,\\
	4&\text{if}\ L(3)\simeq \langle 1\rangle \perp \mathbb{A}.\end{cases}
$$
For each possible candidate of the $4\times 4$ section $L(4)$ of $L$, one may easily find a proper sublattice of $I_2$ which is not represented by it. Hence we have a finite list of candidates of quinary isolations of $I_2$.
 In fact, with the help of computer programming, we may check that there are exactly $231$ quinary $\z$-lattices which do not represent $I_2$, but do represent all proper sublattices of $I_2$ with index $q$ which is less than or equal to $29$. As a sample,  
$$
\langle 17,17\rangle \nra \langle 1\rangle \perp \begin{pmatrix}2&0&0&1\\0&2&1&0\\0&1&3&1\\1&0&1&6\end{pmatrix}\ \ \text{and}\ \ \langle 29,29\rangle \nra \langle 1\rangle \perp \begin{pmatrix}2&0&1&0\\0&2&0&1\\1&0&4&1\\0&1&1&9\end{pmatrix}.
$$

Any candidate $L$ of quinary isolations of $I_2$ is of the form
$$
L\simeq \langle 1\rangle \perp \begin{pmatrix}2&0&a_1&b_1\\0&2&a_2&b_2\\a_1&a_2&a_3&b_3\\b_1&b_2&b_3&b_4\end{pmatrix},
$$
where
$$
(a_1,a_2,a_3)\in \left\{ \begin{array}{l}(1,1,2),(0,1,2),(0,0,2),(1,1,3),(0,1,3),(0,0,3),\\
	(1,1,4),(0,1,4),(0,0,4),(1,1,5),(0,1,5)\end{array}\right\}.
$$
The coefficients $b_i$ are provided in Table \ref{table1}. Among $231$ candidates in Table 1, two quinary quadratic forms marked with asterisk have been proved in \cite{o} to be isolations of $I_2$. In the upcoming sections, we will prove that $14$ marked with dagger are isolations of $I_2$.

Let $L$ be one of the candidates of quinary isolations of $I_2$. To show that $L$ is indeed a quinary isolation of $I_2$, it suffices to show that any sublattice $\ell$ of $I_2$ with index $p$ is represented by $L$ for any prime $p$. 

If $p=2$, then $\ell \simeq \langle 2,2\rangle$ or $\langle 1,4\rangle$. Note that these two binary lattices are represented by all candidates in Table \ref{table1} of quinary isolations of $I_2$. Hence we always assume that $p$ is odd. Note that $\ell_q \simeq \langle 1,1\rangle$ for any prime $q \ne p$.  Assume that $\ell=\begin{pmatrix} a&b\\b&c\end{pmatrix}$ is Minkowski reduced, that is, $0\le 2b\le a \le c$. Then, we have
\begin{equation} \label{mod2}
d\ell=p^2 \ \text{and} \ \begin{cases} 
&a \equiv 2 \Mod 8, \ b\equiv 1 \Mod 2, \ c\equiv 1 \Mod 4 \quad \text{or}\\
&a \equiv 1 \Mod 4, \ b\equiv 1 \Mod 2, \ c\equiv 2 \Mod 8 \quad \text{or}\\
&a \equiv 1 \Mod 4, \ b\equiv 0 \Mod 2, \ c\equiv 1 \Mod 4. \\
\end{cases}
\end{equation}    
If a prime $q(\neq p)$ divides $ac$, then $\left(\frac{-1}{q}\right)=1$ as $ac-b^2=p^2$ and hence $q\not\equiv 3 \Mod{4}$. Thus both $a$ and $c$ are not divisible by any prime $q (\ne p)$ congruent to $3$ modulo $4$.

For a positive integer $t$ and non-negative integers $\alpha$ and $\beta$ with $\alpha^2+\beta^2\neq 0$, we define
\begin{equation}\label{def:ell-tildeell}
\ell(t;\alpha,\beta)=\begin{pmatrix}a-t\alpha^2&b-t\alpha\beta\\b-t\alpha \beta&c-t\beta^2\end{pmatrix},\quad \widetilde{\ell}(t;\alpha,\beta)=\begin{pmatrix}a-t\alpha^2&\frac{b-t\alpha\beta}{2}\\ \frac{b-t\alpha \beta}{2}&\frac{c-t\beta^2}{4}\end{pmatrix}.
\end{equation}
Note that $\ell(t;\alpha,\beta)$ is a subform of $\widetilde{\ell}(t;\alpha,\beta)$ with index 2. We further note that if $\ell(t;\alpha,\beta)\ra M$ for some lattice $M$, then $\ell \ra M\perp \langle t\rangle$.

\section{Isolations of $I_2$ : Basic cases}

In this section, we provide useful proposition and lemma in proving our results. Using these, we prove that six quinary $\z$-lattices $L$ in Table \ref{tablew} are isolations of $I_2$.

\begin{prop} \label{key}
	Let $L$ be a quinary $\z$-lattice and let $M$ be a quaternary $\z$-sublattice of $L$. Let $t$ be a positive integer such that $M\perp \langle t\rangle \ra L$.
	Let $\ell$ and $\ell(t;\alpha,\beta)$ be binary $\z$-lattices given in Section \ref{section2}.
	If $\ell(t;\alpha,\beta)$ is positive definite, then $\ell(t;\alpha,\beta)_q\ra M_q$ for any prime $q$ with $\gcd(q,2dM)=1$.	
	In addition, if $M$ is of class number $1$ and $\ell(t;\alpha,\beta)_q\ra M_q$ for any prime $q$ dividing $2dM$, then $\ell$ is represented by $L$. 
\end{prop}
\begin{proof}
	The last part of the theorem will follow directly once we prove that $\ell(t;\alpha,\beta)_q\ra M_q$ for any prime $q$ with $\gcd(q,2dM)=1$ under the condition that $\ell(t;\alpha,\beta)$ is positive definite.
	 
	Let $q$ be a prime such that $\gcd(q,2dM)=1$. To show that $\ell(t;\alpha,\beta)_q \ra\! M_q$, it suffices to show that $s(\ell(t;\alpha,\beta)_q)=\z_q$, since $M_q$ is unimodular. Assume to the contrary that $s(\ell(t;\alpha,\beta)_q)\subseteq q\z_q$.
	Then
	$$
	a-t\alpha^2\equiv b-t\alpha \beta \equiv c-t\beta^2\equiv 0\Mod q.
	$$
	It follows that $ac-b^2$ is divisible by $q$ and thus $q=p$.
	Put
	$$
	a-t\alpha^2=pm\ \ \text{and}\ \ c-t\beta^2=pn.
	$$
	Noting that
	$$
	\frac34a^2\le \frac34 ac\le ac-b^2=p^2,
	$$
	we have $a\le \displaystyle\frac{2}{\sqrt{3}}p$.
	It follows from
	$$
	0\le t\alpha^2=a-pm\le \left(\frac{2}{\sqrt{3}}-m\right)p
	$$
	that $m=1$.
	Since
	$$
	\frac34 pc\le \frac34 ac\le ac-b^2=p^2,
	$$
	we have $c\le \displaystyle\frac{4}{3}p$.
	So we have $n=1$.
	
	On the other hand, since $b-t\alpha \beta \equiv 0\Mod p$ and
	$$
	d\left(\ell(t;\alpha,\beta)\right)=(a-t\alpha^2)(c-t\beta^2)-(b-t\alpha \beta)^2=p^2-(b-t\alpha \beta)^2>0,
	$$
	it follows that $b=t\alpha \beta$.
	This yields that
	$$
	p^2=ac-b^2=(p+t\alpha^2)(p+t\beta^2)-t^2\alpha^2\beta^2=p^2+p(t\alpha^2+t\beta^2)>p^2,
	$$
	which is absurd. This completes the proof.
\end{proof}

\begin{lem}\label{lem:ell-pos-def}
Under the same notations given above, if either 
\begin{enumerate}
\item $a>\frac{4t}{3}(\alpha^2+\beta^2)$, or
\item $a>(1+u)t\alpha^2$ and $p\ge \frac{4(1+u)}{3u-1}t\beta^2$ for some real number $u>\frac13$,
\end{enumerate}
then $\ell(t;\alpha,\beta)$ is positive definite {\rm (}hence, so is $\widetilde{\ell}(t;\alpha,\beta)${\rm )}.
\end{lem}
\begin{proof}
If $a>\displaystyle\frac{4t}{3}(\alpha^2+\beta^2)$, then
\begin{align*}
d(\ell(t;\alpha,\beta))&=ac-b^2-(t\alpha^2c+t\beta^2a-2t\alpha \beta b)\\
&\ge\frac34 ac-t\alpha^2c-t\beta^2c>0.
\end{align*}
Now, assume that the second condition holds.
Since
$$
a-t\alpha^2>\displaystyle\frac{u}{1+u}a\ \ \text{and}\ \ \frac{4(1+u)}{3u-1}t\beta^2\le p\le c,
$$
we have
\begin{multline*}
d\left(\ell(t;\alpha,\beta)\right)=ac-b^2-(t\alpha^2c+t\beta^2a-2t\alpha \beta b)\\
\ge (a-t\alpha^2)c-b^2-t\beta^2a \ge (a-t\alpha^2)c-\frac{a^2}{4}-t\beta^2a\\
>a\left(\frac{u}{1+u}c-\frac{a}{4}-t\beta^2\right)>\left(\frac{u}{1+u}-\frac14\right)c-t\beta^2\ge0.
\end{multline*}
This completes the proof.
\end{proof}

\begin{rmk}\label{rmk:alpha=0}
If $\alpha=0$ and $p>\frac{4}{3}t\beta^2$, then one may show that $\ell(t;\alpha,\beta)$ is positive definite by taking $u$ sufficiently large in (2) given in Lemma \ref{lem:ell-pos-def}.
\end{rmk}

\begin{thm}\label{thm:basiccase}  
Any quinary $\z$-lattice $L$ in Table \ref{tablew} is an isolation of $I_2$. 
\end{thm}

\begin{proof}  Let $L$ be any quinary $\z$-lattice given in Table \ref{tablew}, and let $M$, $t$, $(\alpha,\beta)$, and $p_0$ be given in the same line with $L$ in Table \ref{tablew}. 

Let $\ell=\begin{pmatrix} a&b\\b&c\end{pmatrix}$ be a Minkowski reduced binary $\z$-sublattice of $I_2$. Recall that $\ell$ satisfies \eqref{mod2}. If $p< p_0$, then one may directly check (with the help of computer programming) that $\ell \ra L$. Now we assume that $p\ge p_0$.

Note that $M \perp \langle t\rangle \ra L$. Note further that $\ell(t;0,\beta)$ is positive definite by Lemma \ref{lem:ell-pos-def} and Remark \ref{rmk:alpha=0}.
Therefore, by Proposition \ref{key}, $\ell \ra L$ if 
$$
\ell(t;0,\beta)_q=\begin{pmatrix} a&b\\b&c-t\beta^2\end{pmatrix} \ra M_q
$$ 
for any prime $q$ contained in the set given in the third column of the same line with $L$ in Table \ref{tablew}.   

Since all the other cases can be done in similar manners, we only consider the sixth case in Table \ref{tablew}. So, we may assume that $q \in \{2,3,5\}$. For $q=3, 5$, since $d(\ell(15;0,1))=p^2-15a \not\equiv 0 \Mod{q}$, the $\z_q$-lattice $\ell(15;0,1)_q$ is unimodular. Hence we have $\ell(15;0,1)_q\ra M_q$.
Assume that $q=2$. Note that
$$
d(\ell(15;0,1)_2)=p^2-15a\equiv \!\!\begin{cases}
	3 \Mod{8} &\!\!\!\!\text{if } a\equiv 2 \!\Mod{8},\\
	2 \Mod{4} &\!\!\!\!\text{if } a\equiv 1 \!\Mod{4}.
	                 \end{cases}
$$
Hence one may easily check that
$$
\ell(15;0,1)_2\simeq \begin{cases}
	\mathbb{A} &\text{if } a\equiv 2 \Mod{8}, \\
	\langle a,a\cdot d(\ell(15;0,1)_2) \rangle &\text{if } a\equiv 1 \Mod{4}.
	\end{cases}
$$
In each case, we have $\ell(15;0,1)_2\ra M_2\simeq \langle 3,3,3,5\rangle$. Therefore we have $\ell(15;0,1)\ra M$ and $\ell \ra M\perp \langle 15\rangle \ra L$. This completes the proof.
\end{proof}


\begin{normalsize}
	\begin{table}[ht]

	\renewcommand{\arraystretch}{1}
		\caption{Data for the proof of Theorem \ref{thm:basiccase}} \label{tablew}

		\begin{tabular}{c|c|c|c|c|c}
			\hline
			$L$ & $M$ & $\{q : q\mid 2dM\}$ & $t$ & $(\alpha,\beta)$  &  $p_0$  \\
			\hline \hline
			$\langle 1\rangle \perp \begin{pmatrix} 2&0&0&1\\ 0&2&1&1\\ 0&1&2&0\\ 1&1&0&3\end{pmatrix}$ & $\langle 1 \rangle \perp \begin{pmatrix} 2&1&1\\ 1&2&0\\ 1&0&3\end{pmatrix}$ & $\{2,7\}$ &  $77$ &  $(0,1)$ & $103$  \\
			\hline	
			$\langle 1,2,3\rangle \perp \begin{pmatrix} 2&1\\1&2\end{pmatrix}$ &$\langle 1,3\rangle \perp \begin{pmatrix} 2&1\\1&2\end{pmatrix}$ & $\{2,3\}$ &  $2$ &  $(0,2)$ & $11$ \\
			\hline		
			$\langle 1,2\rangle \perp \begin{pmatrix} 2&0&1\\0&2&1\\ 1&1&3\end{pmatrix}$ &$\langle 1\rangle \perp \begin{pmatrix} 2&0&1\\0&2&1\\ 1&1&3\end{pmatrix}$ & $\{2\}$ &  $2$ &  $(0,1)$ & $3$ \\
			\hline	
			$\langle 1\rangle \perp \begin{pmatrix} 2&0&1&1\\ 0&2&1&0\\ 1&1&2&0\\ 1&0&0&3\end{pmatrix}$ &$\langle 1\rangle \perp \begin{pmatrix} 2&1&1\\1&2&0\\1&0&3\end{pmatrix}$ & $\{2,7\}$ &  $63$ &  $(0,1)$ & $89$  \\
			\hline	
			$\langle 1,2\rangle \perp \begin{pmatrix} 2&1&1\\1&3&0\\ 1&0&3\end{pmatrix}$ & $\langle 1\rangle \perp \begin{pmatrix} 2&1&1\\1&3&0\\ 1&0&3\end{pmatrix}$& $\{2,3\}$ &  $2$ &  $(0,3)$ & $29$ \\
			\hline	
			$\langle 1\rangle \perp \begin{pmatrix} 2&1\\1&3\end{pmatrix} \perp \begin{pmatrix} 2&1\\1&3\end{pmatrix}$ & $\langle 1,3\rangle \perp \begin{pmatrix} 2&1\\1&3\end{pmatrix}$& $\{2,3,5\}$ &  $15$ &  $(0,1)$ & $23$  \\
			\hline	
			\end{tabular}
		
	\end{table}
\end{normalsize}

\section{Isolations of $I_2$ : General cases}

In this section, we prove that four quinary $\z$-lattices $L(i)$ with $i=1,2,3,4$ in Table \ref{tablegeneral} are isolations of $I_2$.

We define quaternary $\z$-lattices
$$
M(1)=M(2)=\langle 1,2,2,3\rangle,\  M(3)=\begin{pmatrix}2&1\\1&4\end{pmatrix} \perp \begin{pmatrix}2&1\\1&4\end{pmatrix}, \ M(4)=\langle 1\rangle \perp \begin{pmatrix}2&0&1\\0&2&1\\1&1&4\end{pmatrix},
$$
and
$$
N(1)=\langle 1,3\rangle \perp \begin{pmatrix}2&1\\1&4\end{pmatrix},\ N(2)=\langle 1,2\rangle \perp \begin{pmatrix}2&1\\1&3\end{pmatrix},\ N(3)=\langle 1,7\rangle \perp \begin{pmatrix}2&1\\1&4\end{pmatrix},
$$
and put $N(4)=M(4)$.
Note that
$$
d(M(1))=d(M(2))=2^23,\ \ d(M(3))=7^2,\ \ d(M(4))=2^23.
$$


\begin{thm} \label{thmgeneral}
The quinary $\z$-lattice $L(i)$ in Table \ref{tablegeneral} is an isolation of $I_2$ for any $i=1,2,3,4$.
\end{thm}

\begin{proof}
Let $\ell=\begin{pmatrix} a&b\\b&c\end{pmatrix}$ be a Minkowski reduced binary $\z$-sublattice of $I_2$ with index $p$ so that it satisfies \eqref{mod2}.
If $p\le 241$, one may directly check that $L(i)$ represents $\ell$ for any $i=1,2,3,4$.
From now on, we assume that $p>241$.

We first consider the case when $i=1$.
Note that $M(1)\perp \langle 14\rangle \ra L(1)$.
First, assume that $a\equiv 1\Mod 2$.
If we define $\well=\ell(14;0,1)$, then one may easily show that $\well_q\ra M(1)_q$ for $q=2,3$.
Since $\well$ is positive definite by Lemma \ref{lem:ell-pos-def} and Remark \ref{rmk:alpha=0}, $\ell \ra L(1)$ by Proposition \ref{key}. 
Next assume that $a\equiv 0\Mod 2$ and $a\ge 26$.
If we define $\well=\ell(14;1,0)$, then one may show that $\well_q\ra M(1)_q$ for $q=2,3$.
Furthermore since $\well$ is positive definite by Lemma \ref{lem:ell-pos-def},   
$\ell \ra L(1)$ by Proposition \ref{key}.
The case when $a$ is an even integer less than or equal to 24 will be considered later.
For $i=2,3,4$, one may deduce that $\ell \ra L(i)$ by the similar argument using Proposition \ref{key}.
Some data needed to prove the remaining cases are provided in Table \ref{tablegeneral}.
Note that
$$
\ell =\begin{pmatrix}a&b\\b&c\end{pmatrix} \simeq \begin{pmatrix}a-2b+c&b-c\\b-c&c\end{pmatrix}.
$$
Hence if $\well=\begin{pmatrix}a-2b+c&b-c\\b-c&c-35\end{pmatrix}\ra M(3)$, then $\ell \ra L(3)$.
Since we are assuming that $p\ge 241$, $\well$ is positive definite.

\begin{normalsize}
	
	\begin{table}[ht]
		
		\renewcommand{\arraystretch}{1}
		
		\caption{Data for the proof of Theorem \ref{thmgeneral}}
		
		\label{tablegeneral}
		
		\begin{tabular}{c|c|c|c|c|c}
			
			\hline
			
			$i$ & $L(i)$ & $\{q\mid 2dM(i)\}$ & $t$ & $a$ & $\well$\\
			
			\hline
			
			\hline
			
			\multirow{2}{*}[-0.5em]{1} & \multirow{2}{*}[-0.4em]{$\langle 1,2,3\rangle \perp \begin{pmatrix}2&1\\1&4\end{pmatrix}$} & \multirow{2}{*}[-0.5em]{$\{2,3\}$} & \multirow{2}{*}[-0.5em]{$14$} & $a\equiv 1\Mod 2$ & $\ell(14;0,1)$\\
			
			\cline{5-6}
			
			& & & & $\begin{array}{c}a\equiv 0\Mod 2\\ a\ge 26\end{array}$ & $\ell(14;1,0)$ \\
			
			\hline
			
			\multirow{2}{*}[-0.5em]{2} & \multirow{2}{*}[-0.5em]{$\langle 1,2,3\rangle \perp \begin{pmatrix}2&1\\1&5\end{pmatrix}$} & \multirow{2}{*}[-0.5em]{$\{2,3\}$} & \multirow{2}{*}[-0.5em]{$18$} & $a\equiv 1\Mod 2$ & $\ell(18;0,1)$\\
			
			\cline{5-6}
			
			& & & & $\begin{array}{c}a\equiv 0\Mod 2\\ a\ge 26\end{array}$ & $\ell(18;1,0)$\\
			
			\hline
			
			\multirow{3}{*}[-1em]{3} & \multirow{3}{*}[-1em]{$\langle 1\rangle \perp \begin{pmatrix}2&1\\1&3\end{pmatrix} \perp \begin{pmatrix}2&1\\1&4\end{pmatrix}$}& \multirow{3}{*}[-1em]{$\{2,7\}$} & \multirow{3}{*}[-1em]{$35$} & $a\equiv 0\Mod 2$ & $\ell(35;0,1)$\\
			
			\cline{5-6}
			
			& & & & $\begin{array}{c}c\equiv 0\Mod 2\\ a\ge 37\end{array}$ & $\ell(35;1,0)$ \\
			
			\cline{5-6}
			
			& & & & $\begin{array}{c}ac\equiv 1\Mod 2\\ a\ge 37\end{array}$ & $\begin{pmatrix}a-2b+c&b-c\\b-c&c-35\end{pmatrix}$ \\
			
			\hline
			
			\multirow{3}{*}[-1em]{4} & \multirow{3}{*}[-1em]{$\langle 1,5\rangle \perp \begin{pmatrix}2&0&1\\0&2&1\\1&1&4\end{pmatrix}$} & \multirow{3}{*}[-1em]{$\{2,3\}$} & \multirow{3}{*}[-1em]{$5$} & $a\equiv 0\Mod 2$ & $\ell(5;0,1)$\\
			
			\cline{5-6}
			
			& & & & $\begin{array}{c}c\equiv 0\Mod 2\\ a\ge 7\end{array}$ & $\ell(5;1,0)$\\
			
			\cline{5-6}
			
			& & & & $\begin{array}{c}ac\equiv 1\Mod 2\\ a\ge 13\end{array}$ & $\begin{pmatrix}a-2b+c&b-c\\b-c&c-5\end{pmatrix}$ \\
			
			\hline
			
		\end{tabular}
		
	\end{table}
	
\end{normalsize}

Now, we consider the exceptional cases, which are
$$
(i,a)\in \left\{ \begin{array}{c} (1,2),(1,10),(2,2),(2,10),(3,1),(3,5),\\ (3,13),(3,17),(3,25),(3,29),(4,1),(4,5)\end{array} \right\}.
$$
Assume first that $(i,a)\in \{(1,2),(2,2)\}$.
Define positive integers $n_1$ and $n_2$ by
$$
n_1=\displaystyle\frac{p^2+1}{2}-(i+3)\ \ \text{and}\ \ n_2=\displaystyle\frac{p^2+1}{2}-(9(i+3)-4).
$$
Then either $n_1$ or $n_2$ is not of the form $2^{2u+1}(8v+5)$ for any nonnegative integers $u$ and $v$.
Hence either $n_1$ or $n_2$ is represented by $\langle 1,2,3\rangle$.
Since
$$
\begin{pmatrix}2&1\\1&9(i+3)-4\end{pmatrix}=\begin{pmatrix}1&0\\-1&3\end{pmatrix}\begin{pmatrix}2&1\\1&i+3\end{pmatrix}\begin{pmatrix}1&-1\\0&3\end{pmatrix},
$$
it follows that
$$
\ell=\begin{pmatrix}2&1\\1&\frac{p^2+1}{2}\end{pmatrix}\ra \langle 1,2,3\rangle \perp \begin{pmatrix}2&1\\1&i+3\end{pmatrix}=L(i).
$$
Assume that $(i,a)=(4,1)$.
Since $\langle 2,2,5\rangle$ represents all squares of integers except for 1 (see \cite{jko}), we have
$$
\ell=\langle 1,p^2\rangle \ra \langle 1,2,2,5\rangle \ra L(4).
$$
Assume that $(i,a)=(2,10)$.
In this case, one may easily show by using Proposition \ref{key} that
$$ 
\begin{pmatrix}10&b\\b&c-2\cdot 2^2\end{pmatrix}\ra M=\langle 1,3 \rangle \perp \begin{pmatrix} 2 & 1 \\ 1&5 \end{pmatrix}.
$$
Hence $\ell \ra M\perp \langle 2\rangle=L(2)$.

Now, we consider the remaining cases, which are
$$
(i,a)\in \{(1,10),(3,1),(3,5),(3,13),(3,17),(3,25),(3,29),(4,5)\}.
$$
For any $(i,a)\in \{(3,5),(3,17),(3,29)\}$, we define an integer \begin{equation}\label{eqn:defnk}
k=k(i,a,p)\in \{1,2,4\} \text{ such that } p^2+ak^2\equiv 0\Mod 9.
\end{equation}
Assume that $(i,a)=(3,5)$.
In this case, we put
$$
t=35,\ \alpha=0,\ \beta=k,\ N=\langle 1\rangle \perp \begin{pmatrix}2&1\\1&4\end{pmatrix}\perp \langle 7\rangle, \text{ and } \well=\ell(t;\alpha,\beta).
$$
Since $N\perp \langle t\rangle \ra L(i)$, it suffices to show that $\well \ra N$.
Using \cite[Corollary 2.5]{jk} with $A=a-t\alpha^2$, $s=9$, one sees that any (positive definite) binary $\z$-lattice $\begin{pmatrix}A&B \\ B&C \end{pmatrix}$ with $(B,C)\in \mathfrak{S}_s+s\z^2$ is represented by $N$, where
$$
\mathfrak{S}_s=\{(0,0),(1,2),(2,8),(3,0),(4,5),(5,5),(6,0),(7,8),(8,2)\}.
$$
Hence $\well$ is represented by $N$ if $\well$ is positive definite and
$$
(b,c-t\beta^2)=\left(b,\frac{p^2+b^2}{a}-t\beta^2\right) \in \mathfrak{S}_s,
$$
which can easily be verified.

Since each of the other seven pairs $(i,a)$ can be dealt with along the same line with different numbers $t,\alpha,\beta,s$ and a $\z$-lattice $N$,
we only provide those data together with the set $\mathfrak{S}_s$ for each pair in Table \ref{thmgeneral1}.
One may easily check that $(b-t\alpha \beta, c-t\beta^2) \in \mathfrak{S}_s+s\z^2$ so that
$$
\begin{pmatrix}a-t\alpha^2&b-t\alpha \beta \\ b-t\alpha \beta& c-t\beta^2\end{pmatrix} \ra N(i)
$$
by \cite[Corollary 2.5]{jk}.
This implies that $\ell \ra N(i)\perp \langle t\rangle \ra L(i)$.

The case when $(i,a)=(1,10)$ and $b=3$ should be considered in a slightly different way. 
Instead of $\begin{pmatrix}10&3\\3&\frac{p^2+9}{10}\end{pmatrix}$,
we consider the Gram matrix
$$
\begin{pmatrix}10&13\\13&\frac{p^2+169}{10}\end{pmatrix}=
\begin{pmatrix}1&0\\1&1\end{pmatrix}
\begin{pmatrix}10&3\\3&\frac{p^2+9}{10}\end{pmatrix}
\begin{pmatrix}1&1\\0&1\end{pmatrix}.
$$
In fact, one may easily show that
$$
\left(13,\frac{p^2+169}{10}\right)\in (1,2)+3\z^2
$$
and $(1,2)\in \mathfrak{S}_s=\{(0,0),(1,2),(2,2)\}$
so that 
$$
\begin{pmatrix}10-2&13\\13&\frac{p^2+169}{10}\end{pmatrix}\ra N(1)
$$
by \cite[Corollary 2.5]{jk}.
This implies that $\ell \ra N(1)\perp \langle 2\rangle \ra L(1)$.
\end{proof}

\begin{normalsize}
\begin{table}[ht]
\renewcommand{\arraystretch}{1.05}
\caption{Data for the exceptional cases of Theorem \ref{thmgeneral}}
\label{thmgeneral1}
\begin{tabular}{|c|c|c|c|c|c|c|}
\hline
\multirow{2}{*}{$(i,a)$} & $t$ & $\alpha$ & $\beta$ & $s$ & $\well$ & $N$ \\
\cline{2-7}
&\multicolumn{6}{|c|}{$\mathfrak{S}_s$}\\
\hline
\hline
\multirow{2}{*}{$(1,10)$} & 2 & 1 & 0 & 3 & $\well=\begin{pmatrix}8&1\\1&\frac{p^2+1}{10}\end{pmatrix},\ \begin{pmatrix}8&13\\13&\frac{p^2+169}{10}\end{pmatrix}$ & $N(1)$\\
\cline{2-7}
&\multicolumn{6}{|c|}{$\{(0,0),(1,2),(2,2)\}$}\\
\hline
\hline
\multirow{2}{*}{$(3,1),(3,13),(3,25)$} & 14 & 0 & 1 & 6 & $\ell(t;\alpha,\beta)$ & $N(2)$\\
\cline{2-7}
&\multicolumn{6}{|c|}{$(\z/6\z)^2\setminus \{(0,1),(1,2),(2,5),(3,4),(4,5),(5,2)\}$}\\
\hline
\hline
\multirow{2}{*}{$(3,5)$} & 35 & 0 & $k^\ast$ & $9$ & $\ell(t;\alpha,\beta)$ & $N(3)$\\
\cline{2-7}
&\multicolumn{6}{|c|}{$\{(0,0),(1,2),(2,8),(3,0),(4,5),(5,5),(6,0),(7,8),(8,2)\}$}\\
\hline
\hline
\multirow{2}{*}{$(3,17)$} & 35 & 0 & $k^\ast$ & $9$ & $\ell(t;\alpha,\beta)$ & $N(3)$\\
\cline{2-7}
&\multicolumn{6}{|c|}{$\{(0,0),(1,8),(2,5),(3,0),(4,2),(5,2),(6,0),(7,5),(8,8)\}$}\\
\hline
\hline
\multirow{2}{*}{$(3,29)$} & 35 & 0 & $k^\ast$ & $9$ & $\ell(t;\alpha,\beta)$ & $N(3)$\\
\cline{2-7}
&\multicolumn{6}{|c|}{$\{(0,0),(1,5),(2,2),(3,0),(4,8),(5,8),(6,0),(7,2),(8,5)\}$}\\
\hline
\hline
\multirow{2}{*}{$(4,5)$} & 5 & 0 & 2 & 3 & $\ell(t;\alpha,\beta)$ & $N(4)$\\
\cline{2-7}
&\multicolumn{6}{|c|}{$\{(0,0),(1,2),(2,2)\}$}\\
\hline
\multicolumn{7}{r}{$\ast$ See \eqref{eqn:defnk} for the definition of $k$.}
\end{tabular}
\end{table}
\end{normalsize}
\section{Sublattices represented by a lattice in a genus }

In this section, we introduce some method on the representations of binary $\z$-lattices by the genera of some quaternary $\z$-lattices. Using this method, we prove that four quinary $\z$-lattices $L$ in Table \ref{tableg} are isolations of $I_2$.

\begin{lem} \label{2-core}
Let $M$ be any quaternary $\z$-lattice given in the second column in Table \ref{tableg}, and let $\widetilde{M} \in \gen(M)$ be any quaternary $\z$-lattice given in the third column of the same line with $M$ in Table \ref{tableg}.
Let $N$ be a $\z$-sublattice of $\widetilde{M}$ such that $\widetilde{M}/N \simeq \z/2\z\oplus\z/2\z\oplus\z/2\z$. Then $N$ is represented by $M$.
\end{lem}

\begin{proof}
Let $\{\bx_i\}$ be a basis for $\widetilde{M}$. Then there are $\epsilon_i \in \{0,1\}$ such that $(\epsilon_1,\dots,\epsilon_4)\ne (0,0,0,0)$ and 
$$
N=\z\left(\sum_{i=1}^4 \epsilon_i\bx_i\right)+\z 2\bx_1+\z 2\bx_2+\z 2\bx_3+\z 2\bx_4.
$$
Now, one may directly check that $N$ is represented by $M$ for any possible $\epsilon_i$'s for $i=1,2,3,4$.  
\end{proof}

\begin{prop} \label{key2}
Let $L$ be a quinary $\z$-lattice and let $M$ be a quaternary $\z$-sublattice of $L$. Let $t$ be a positive integer such that $M\perp \langle t\rangle \ra L$.
Let $\ell$, $\ell(t;\alpha,\beta)$, and $\widetilde{\ell}(t;\alpha,\beta)$  be binary $\z$-lattices given in Section \ref{section2}.
Suppose that
\begin{enumerate}
  
\item $\widetilde{\ell}(t;\alpha,\beta)$ is positive definite,

\item $\widetilde{\ell}(t;\alpha,\beta)_q\ra M_q$ for any prime $q\mid 2dM$, 
 
\item for any $\z$-lattice $\widetilde{M} \in \gen(M)$ and for any $\z$-sublattice $N$ of $\widetilde{M}$ such that $\widetilde{M}/N \simeq \z/2\z\oplus\z/2\z\oplus\z/2\z$, $N$ is represented by $M$.
\end{enumerate}
Then $\ell$ is represented by $L$. 
\end{prop}
\begin{proof}
As in the proof of Proposition \ref{key}, one may prove that if $\widetilde{\ell}(t;\alpha,\beta)$ is positive definite, then $\widetilde{\ell}(t;\alpha,\beta)_q\ra M_q$ for any prime $q$ with $\gcd(q,2dM)=1$.
Hence by $(1)$ and $(2)$, the binary $\z$-lattice $\widetilde{\ell}(t;\alpha,\beta)$ is represented by some $\z$-lattice $\widetilde{M}\in \gen(M)$.  Since $\ell(t;\alpha,\beta)$ is a sublattice of $\widetilde{\ell}(t;\alpha,\beta)$ with index 2, $\ell(t;\alpha,\beta)$ is represented by some $\z$-sublattice $N$ of $\widetilde{M}$ such that $\widetilde{M}/N \simeq \z/2\z\oplus\z/2\z\oplus\z/2\z$. Hence $\ell(t;\alpha,\beta)$ is represented by $M$ by $(3)$, which implies that $\ell$ is represented by $L$.
\end{proof}

\begin{thm}   
Any quinary $\z$-lattice $L$ in Table \ref{tableg} is an isolation of $I_2$. 
\end{thm}

\begin{proof} 
Let $L$ be any quinary $\z$-lattice given in Table \ref{tableg}, and let $M$, $\widetilde{M}$, $t$, $(\alpha, \beta)$, and $p_0$ be given in the same line with $L$ in Table \ref{tableg}. 

Let $\ell=\begin{pmatrix} a&b\\b&c\end{pmatrix}$ be a Minkowski reduced binary $\z$-sublattice of $I_2$ with index $p$. Recall that $\ell$ satisfies \eqref{mod2}. If $p < p_0$, then one may directly check (with the help of computer programming) that $\ell \ra L$. Now we assume that $p\ge p_0$. Since $p^2=ac-b^2 \le c^2$, we have $c \ge p_0$.

Note that $M \perp \langle t\rangle = L$. Note further that $\widetilde{\ell}(t;0,\beta)$ is positive definite by Lemma \ref{lem:ell-pos-def} and Remark \ref{rmk:alpha=0}. Therefore, by Lemma \ref{2-core} and Proposition  \ref{key2}, $\ell \ra L$ if we take a suitable  $(\alpha,\beta)$ in the same line with $L$ in Table \ref{tableg} such that
$$
\widetilde{\ell}(t;\alpha,\beta)_q=\begin{pmatrix}a-t\alpha^2&\frac{b-t\alpha\beta}{2}\\ \frac{b-t\alpha \beta}{2}&\frac{c-t\beta^2}{4}\end{pmatrix} \ra M_q
$$ 
for any prime $q$ in the set given in the same line with $L$ in Table \ref{tableg}.   

Since all the other cases can be done in similar manners, we only consider the first case in Table \ref{tableg}. First, assume that  $a \equiv c \equiv 1 \Mod 4$. Since $\widetilde{\ell}(5;\alpha,\beta)_3$ is unimodular or  $\widetilde{\ell}(5;\alpha,\beta)_3 \simeq \langle -1,\eta\rangle$ for some $\eta \in \z_3$, we have $\widetilde{\ell}(5;\alpha,\beta)_3 \ra M_3$ for any $(\alpha,\beta) \in \{(0,1), (0,3), (0,5)\}$. 
If $a \equiv 1 \Mod 8$, then $\widetilde{\ell}(5;0,1)_2$ is a binary unimodular $\z_2$-lattice representing $1$. Hence it is represented by $M_2=\langle 1,1,1,6\rangle$. Now, assume that $a\equiv 5 \Mod 8$. Then one may easily check that
$$
\begin{cases}  &\widetilde{\ell}(5;0,3) \ra M_2 \quad \text{if $d\widetilde{\ell}(5;0,3) \ne 2^{2s+1}(8r+5)$ for any integers $s,r$,}\\
                       &\widetilde{\ell}(5;0,1) \ra M_2 \quad \text{if $d\widetilde{\ell}(5;0,3)=2^{2s+1}(8r+5)$ for some $s \ne 1$ and $r$,}\\
                        &\widetilde{\ell}(5;0,5) \ra M_2 \quad \text{if $d\widetilde{\ell}(5;0,3)=2^{3}(8r+5)$ for some $r$.}\\
\end{cases}
$$
Therefore $\well(5;\alpha,\beta)\ra \gen(M)$ for some $(\alpha,\beta)\in\{(0,1), (0,3), (0,5)\}$, which implies that $\ell \ra L$ by Proposition \ref{key2}.

Assume that $a \equiv 1 \Mod 4$ and $c \equiv 2 \Mod 8$. Note that 
$$
\ell=\begin{pmatrix} a&b\\b&c\end{pmatrix} \simeq \begin{pmatrix} a&a-b\\a-b&a-2b+c\end{pmatrix}
$$
and $a-2b+c \equiv 1 \Mod 4$. In this case, we consider 
$$
\ell=\begin{pmatrix} a&a-b\\a-b&a-2b+c\end{pmatrix}
$$ 
instead of $\begin{pmatrix} a&b\\b&c\end{pmatrix}$. Similarly to the above, one may show that 
$$
\begin{pmatrix} a&\frac{a-b}2\\\frac{a-b}2&\frac{a-2b+c-5\beta^2}4\end{pmatrix} \ra \gen(M) \quad \text{for some $\beta \in \{1,3,5\}$}.
$$
This implies that $\ell \ra L$.

Finally, assume that $a \equiv 2 \Mod 8$ and $c\equiv 1 \Mod 4$. Note that 
$$
\ell \simeq \begin{pmatrix} a+2b+c&b+c\\b+c&c\end{pmatrix}\simeq \begin{pmatrix} a-2b+c&b-c\\b-c&c\end{pmatrix}.
$$
In this case, we consider 
$$
\ell=\begin{pmatrix} a+2b+c&b+c\\b+c&c\end{pmatrix} \quad \text{or}\quad \begin{pmatrix} a-2b+c&b-c\\b-c&c\end{pmatrix}
$$ 
instead of $\begin{pmatrix} a&b\\b&c\end{pmatrix}$.
Let
$$
\ell_+ = \begin{pmatrix} a+2b+c&b+c\\b+c&c-5\end{pmatrix} \quad \text{and} \quad \ell_-=\begin{pmatrix} a-2b+c&b-c\\b-c&c-5\end{pmatrix}.
$$
By the similar reasoning as above, we have 
$$
\widetilde{\ell_{+}}=\begin{pmatrix} a+2b+c&\frac{b+c}2\\ \frac{b+c}2&\frac{c-5}4\end{pmatrix} \ra M_2 \quad \text{or} \quad \widetilde{\ell_{-}}=\begin{pmatrix} a-2b+c&\frac{b-c}2\\ \frac{b-c}2&\frac{c-5}4\end{pmatrix} \ra M_2
$$
according as $a+2b+c\equiv 1 \Mod{8}$ or $a-2b+c\equiv 1\Mod{8}$, and $\widetilde{\ell_{\pm}} \ra M_3$.
Furthermore, if $a \ge 21$, then we have
$$
4\cdot d\widetilde{\ell_{\pm}}=ac-b^2-5(a\pm2b+c) \ge \frac {3ac}4-15c>0.
$$
Assume that $a=10$. Since $c\ge 167$, we have
$$
4\cdot d\widetilde{\ell_{\pm}}=ac-b^2-5(a\pm2b+c) \ge (a-5)c-25-50-50>0.
$$
Therefore either $\widetilde{\ell_{+}}$ or  $\widetilde{\ell_{-}}$ is represented by the genus of $M$ and hence $\ell \ra L$.  
Finally, for the remaining case when $a=2$, one may directly show that 
$$
\begin{pmatrix} 2&1\\1&\frac{1+p^2}2\end{pmatrix} \ra L=\langle 1,2,5\rangle \perp \begin{pmatrix} 2&1\\1&2\end{pmatrix}.
$$
This completes the proof. 
\end{proof}


\begin{normalsize}
	\begin{table}[ht]
	     \renewcommand{\arraystretch}{0.95}  
		\caption{Data for Proposition \ref{key2}} 		\label{tableg}
		\begin{tabular}{c|c|c|c|c|c|c}
			\hline
			$L$ & $M$ & $\gen(M)\setminus\{M\}$ & $\{q\mid 2dM\}$ & $t$ & $(\alpha,\beta)$  &  $p_0$  \\
			\hline \hline
			$\langle 1,2,5\rangle \perp \begin{pmatrix} 2&1\\1&2 \end{pmatrix}$     & $\langle 1,2\rangle \perp \begin{pmatrix} 2&1\\1&2 \end{pmatrix}$ &  $\langle 1,1,1,6\rangle$   & $\{2,3\}$ &  $5$ &  $\begin{array}{c} (0,1)\\ (0,3)\\ (0,5)\end{array}$ & $167$  \\
			\hline	
			$\langle 1,2,9\rangle \perp \begin{pmatrix} 2&1\\1&2 \end{pmatrix}$     & $\langle 1,2\rangle \perp \begin{pmatrix} 2&1\\1&2 \end{pmatrix}$ &  $\langle 1,1,1,6\rangle$   & $\{2,3\}$ &  $9$ &  $(0,1)$ & $211$ \\
			\hline		
			$\langle 1,2,9\rangle \perp \begin{pmatrix} 2&1\\1&3 \end{pmatrix}$     & $\langle 1,2\rangle \perp \begin{pmatrix} 2&1\\1&3 \end{pmatrix}$ &  $\langle 1,1,1,10\rangle$     & $\{2,5\}$ &  $9$ &  $\begin{array}{c} (0,1)\\ (0,3)\\ (0,5)\end{array}$ & $307$ \\
			\hline
	$\langle 1,2,5\rangle \perp \begin{pmatrix} 2&1\\1&5 \end{pmatrix}$  & $\langle 1,2\rangle \perp \begin{pmatrix} 2&1\\1&5 \end{pmatrix}$ &  $\begin{array}{c}\langle 1,1,1,18\rangle, \\ \begin{pmatrix} 2&1&1&0\\ 1&2&0&0\\ 1&0&3&1\\ 0&0&1&3\end{pmatrix}  \end{array}$   & $\{2,3\}$ &  $5$ &  $(0,1)$ & $29$  \\		

			 \hline  
			\end{tabular}
	\end{table}
\end{normalsize}



\begin{thebibliography}{abcd}

\bibitem{c} J. W. S. Cassels, {\em Rational quadratic forms}, Academic Press, London, 1978.

\bibitem{co} W.K. Chan and B.-K. Oh, {\em Can we recover an integral quadratic form by representing all its subforms?}, preprint, arXiv:2201.08957 (2022).

\bibitem{ekk} N. D.  Elkies, D. M. Kane, and S. D. Kominers,  {\em Minimal S-universality criteria may vary in size}, J. Th\'{e}or. Nombres Bordeaux  \textbf{25}(2013),  557--563.

\bibitem{loy} I. Lee, B.-K. Oh, H. Yu, {\em A finiteness theorem for positive definite almost $n$-regular quadratic forms}, J. Ramanujan Math. Soc. \textbf{35}(2020), no. 1, 81--94.

\bibitem{jk} J. Ju and D. Kim, {\em The pentagonal theorem of sixty-three and generalizations of Cauchy's lemma}, Forum Math. (2023), http://doi.org/10.1515/forum-2023-0014

\bibitem{jko} Y.-S. Ji, M.-K. Kim, and B.-K. Oh, {\em Positive definite quadratic forms representing integers of the form $an^2+b$}, Ramanujan J. \textbf{27}(2012), 329--342.

\bibitem {ki} Y. Kitaoka, {\em Arithmetic of quadratic forms}, Cambridge University Press, 1993.

\bibitem{klo} K. Kim, J. Lee, and B.-K. Oh, {\em Minimal university criterion sets on the representations of binary quadratic forms},  J. Number Theory \textbf{238}(2022),  37--59.

\bibitem{o} B.-K. Oh, {\em Isolations of cubic lattices from their proper sublattices}, Adv. Math. \textbf{430} (2023), Paper No. 109210.


\bibitem{om2} O. T. O'Meara, {\em Introduction to quadratic forms}, Springer-Verlag, New York, 1963.

\end{thebibliography}
\end{document}